\documentclass[11pt, a4paper, twoside]{amsart}

\usepackage{amsmath}
\usepackage{amssymb}
\usepackage{amsthm}

\usepackage{graphicx}
\usepackage{subcaption}

\usepackage{xcolor}

\usepackage{todonotes}

\newcommand{\R}{\mathbb{R}}
\newcommand{\Z}{\mathbb{Z}}
\newcommand{\N}{\mathbb{N}}

\newcommand{\K}{\mathcal{K}}
\newcommand{\Ko}{\mathcal{K}_o}
\newcommand{\conv}{\mathrm{conv}}

\newcommand{\vol}{\mathrm{vol}}
\newcommand{\suk}{\mathrm{h}}

\newcommand{\va}{{\boldsymbol a}}
\newcommand{\vb}{{\boldsymbol b}}
\newcommand{\vc}{{\boldsymbol c}}

\newcommand{\ve}{{\boldsymbol e}}

\newcommand{\vt}{{\boldsymbol t}}

\newcommand{\vu}{{\boldsymbol u}}
\newcommand{\vv}{{\boldsymbol v}}
\newcommand{\vw}{{\boldsymbol w}}
\newcommand{\vx}{{\boldsymbol x}}
\newcommand{\vy}{{\boldsymbol y}}

\newcommand{\vnull}{{\boldsymbol 0}}

\newcommand{\LE}{\#}

\usepackage{hyperref}
\hypersetup{backref, pdfpagemode=FullScreen, colorlinks=true,
  citecolor=magenta, linkcolor=cyan, urlcolor=blue}

\title[Bounds on the lattice point enumerator]{Bounds on the lattice point enumerator via slices and projections}
\author{Ansgar Freyer and Martin Henk}
\address{Technische Universität Berlin, Institut für Mathemtik, Sekr. MA4-1, Stra{\ss}e des 17 Juni 136, D-10623 Berlin}
\email{freyer@math.tu-berlin.de, henk@math.tu-berlin.de}
\date{}

\numberwithin{equation}{section}

\usepackage{mathtools}
\mathtoolsset{showonlyrefs=true}

\begin{document}


\theoremstyle{plain}
\newtheorem{theorem}{Theorem}[section]
\newtheorem{lemma}[theorem]{Lemma}

\newtheorem{proposition}[theorem]{Proposition}
\newtheorem{example}[theorem]{Example}
\newtheorem*{question}{Question}
\newtheorem{thmx}{Theorem}
\renewcommand{\thethmx}{\Alph{thmx}}
\newtheorem{lemmax}[thmx]{Lemma}

\theoremstyle{definition}
\newtheorem*{definition}{Definition}

\newtheorem{remark}[theorem]{Remark}

\begin{abstract} Gardner, Gronchi and Zong posed the problem to 
  find a discrete analogue of M.~Meyer's inequality bounding
   the volume of a convex body from
   below by the geometric mean of the volumes of its slices with the
   coordinate hyperplanes. Motivated by this problem, for which we
   provide a first general bound, we study in a more general context
   the question to bound the number of lattice points of a convex body
   in terms of slices as well as projections.  
\end{abstract}

\maketitle


\section{Introduction}
\label{sec:intro}
One of the central questions in 
Geometric Tomography is to determine or to reconstruct a set $K$ in
the $n$-dimensional Euclidean space $\R^n$ by some of its lower
dimensional ``structures'' (see \cite{Gardner2006}). Usually, these are projections on and sections with lower dimensional subspaces of $\R^n$.
A classical and very well-known example in this context is the famous
Loomis-Whitney inequality \cite{lw}, which compares the volume of a
non-empty compact set $K$ to the geometric mean of its projections onto the coordinate hyperplanes:
\begin{equation}
\label{eq:lw}
\vol(K)^{\frac{n-1}{n}} \leq \Big(\prod_{i=1}^n \vol_{n-1} (K|\ve_i^\perp)\Big)^\frac{1}{n}.
\end{equation}
Here $\vol(K)$ denotes the volume, i.e., the $n$-dimensional Lebesgue
measure of the set $K$, and  $\vol_{n-1} (K|\ve_i^\perp)$ the
$(n-1)$-dimensional volume of the orthogonal projection of $K$ onto
the coordinate hyperplane orthogonal to the $i$th unit vector $\ve_i$. 
Equality is attained, e.g., if $K=[a_1,b_1]\times...\times [a_n,b_n]$,
$a_i,b_i\in\R$, is a rectangular box. For various generalizations and
extensions of this inequality we refer to
\cite{BrazitikosGiannopoulosLiakopoulos2018} and the references within.

Loomis and Whitney proved \eqref{eq:lw}  by  observing that
it suffices to prove it when $K$ is the  pairwise non-overlapping disjoint
union of equal cubes which is then a purely combinatorial problem. In
particular, this combinatorial version implies (and is actually
equivalent to) the following discrete
variant of \eqref{eq:lw}
\begin{equation}
  \label{eq:lw_lat}
 \LE(K)^{\frac{n-1}{n}} \leq \Big(\prod_{i=1}^n \LE(K|\ve_i^\perp)\Big)^\frac{1}{n},
\end{equation}
where  $\LE(M)=|M\cap\Z^n|$ is the lattice point enumerator (with
respect to $\Z^n$). So \eqref{eq:lw} and \eqref{eq:lw_lat}
are equivalent statements for compact sets. The discrete version
\eqref{eq:lw_lat} was also independently proven by Schwenk and Munro \cite{SchwenkMunro1983}.

Due to the comparison of $n$- and $(n-1)$-dimensional volumes in
\eqref{eq:lw}, it is easy to see that there is no lower bound on the volume
in terms of the geometric mean of $\vol_{n-1}
(K|\ve_i^\perp)$. However, if we further assume  that $K\in \K^n$,
i.e., it belongs to the family of convex and compact sets,  and
if we replace projections by sections, then it was shown by M.~Meyer
\cite{meyer} 
\begin{equation}
\label{eq:meyer}
\vol(K)^{\frac{n-1}{n}} \geq \frac{n!^{\frac{1}{n}}}{n}\Big( \prod_{i=1}^n \vol_{n-1}(K\cap \ve_i^\perp)\Big)^{\frac{1}{n}},
\end{equation}
where equality is attained, if and only if $K$ is a  generalized
crosspolytope, i.e., $K=\conv\{a_1\ve_1, -b_1\ve_1,...,a_n \ve_n,
-b_n\ve_n\},$ for some $a_i,b_i\geq 0$. Observe that $n!^{1/n}/n$ is
asymptotically $1/\mathrm{e}$. 

In \cite{ggz}, Gardner, Gronchi and Zong posed the 
question to find a discrete  analogue of M.~Meyer's inequality \eqref{eq:meyer}; more
precisely, they asked 
\begin{question}
Let $n\in\N$. Is there a constant $c_n>0$ such that for all $K\in\K^n$ 
\begin{equation}
\label{eq:dmeyer-question}
\LE(K)^{\frac{n-1}{n}}\geq c_n \Big(\prod_{i=1}^n\LE(K \cap
\ve_i^\perp)\Big)^\frac{1}{n} \,{\huge ?}
\end{equation}
\end{question}
As in the case of the Loomis-Whitney inequality, a discrete version
\eqref{eq:dmeyer-question} would imply the analogous inequality for the
volume, and hence, by  \eqref{eq:meyer} we certainly have $c_n\leq
n!^{1/n}/n$ (cf.\ \eqref{eq:vol-appr}). In the plane, Gardner et al.~\cite{ggz} proved 
\begin{equation}
\label{eq:ggz-dim2}
\LE(K)^{\frac{1}{2}}>\frac{1}{\sqrt{3}} \big(\LE(K\cap \ve_1^\perp)\cdot \LE(K\cap \ve_2
^\perp)\big)^\frac{1}{2},
\end{equation}
for any $K\in\K^2$. The elongated cross-polytope  $K=\conv\{\pm
\ve_1, \pm h\ve_2\}$ shows that $\frac{1}{\sqrt{3}}$ is
asymptotically best possible, i.e., for $h\to\infty$. Hence, in
contrast to the Loomis-Whitney inequality, \eqref{eq:meyer} has no equivalent
discrete version, since in the plane the constant in Meyer's
inequality is $1/\sqrt{2}$. 

The short answer to the question above by Gardner et al.~
is ``No!'' for  arbitrary convex bodies and $n\geq 3$ (cf.~Proposition
\ref{prop:simplex}). Restricted to 
the set $\Ko^n$ of origin-symmetric convex bodies, however, we have
the following result.

\begin{theorem}
\label{thm:dmeyer}
Let $K\in\K_o^n$. Then
\begin{equation*} 
\LE(K)^{\frac{n-1}{n}} > \frac{1}{4^{n-1}} \Big(\prod_{i=1}^n \LE
(K\cap \ve_i^\perp)\Big)^\frac{1}{n}.
\end{equation*} 
\end{theorem}
We do not believe that this inequality is best possible.
Instead, we propose ${3}^\frac{1-n}{n}$ to be the right constant (cf.\ Example \ref{ex:berg}).

While it is not possible to bound the volume of a symmetric convex set
$K$ from above in terms of  $\prod_{i=1}^n \vol_{n-1}(K\cap
\ve_i^\perp)$, Feng, Huang and Li proved in \cite{fhl} that there is a
constant $\tilde c_n\leq (n-1)!$ such that for any $K\in\K_o^n$ there exists an orthogonal basis $\vu_1,...,\vu_n\in\R^n$ such that:
\begin{equation}
\label{eq:rev_meyer}
\vol(K)^{\frac{n-1}{n}} \leq \tilde c_n \Big( \prod_{i=1}^n \vol_{n-1}(K\cap \vu_i^\perp)\Big)^{\frac{1}{n}}.
\end{equation}

In a recent preprint, Alonso-Guti\'{e}rrez and Brazitikos \cite{agb}
improved
this result  considerably: they showed that up to a universal
constant the best possible $\tilde c_n$ is equal to the isotropic
constant $L_K$ which is bounded from above by $n^{1/4}$ (see
\cite{klartag}). Moreover, they proved that this is valid for any
centered convex body, i.e., for a convex body whose centroid is at the origin. Inspired by this, we prove the following inequaltities:

\begin{theorem}
\label{thm:drev_meyer}
Let $K\in\K_o^n$. 
There exists a basis $\vb_1,...,\vb_n$ of the
lattice $\Z^n$ such that 
\begin{equation}\label{eq:rev-meyer-hom}
\LE(K)^{\frac{n-1}{n}} < O(n^2\,2^n) \Big(\prod_{i=1}^n \LE(K\cap \vb_i^\perp) \Big)^{\frac{1}{n}}, 
\end{equation}
and there exists $\vt_i\in\Z^n$, $1\leq i\leq n$,  such that 
\begin{equation}\label{eq:rev-meyer-inhom}
\LE(K)^{\frac{n-1}{n}} < O(n^2) \Big(\prod_{i=1}^n \LE\big(K\cap (\vt_i+\vb_i^\perp)\big) \Big)^{\frac{1}{n}}.
\end{equation}
\end{theorem}

Observe that due to Brunn's concavity principle (see, e.g.,
\cite[Theorem 1.2.1]{ArtsteinAvidanGiannopoulosMilman2015}) the volume
maximal slice $\vol_{n-1}(K\cap(\vt+\vu^\perp))$, $\vt\in\R^n$,  of an
origin-symmetric convex body $K\in\Ko^n$ is always the central slice,
i.e., $\vt=\vnull$. This is no longer true regarding lattice points
which explains the difference bewtween \eqref{eq:rev-meyer-hom} and \eqref{eq:rev-meyer-inhom}. Here we
have the following kind of a discrete Brunn's concavity principle.
\begin{lemma} Let $K\in\Ko^n$ and let $L\subset\R^n$ be a
  $k$-dimensional linear lattice subspace, i.e., $\dim(L\cap\Z^n)=k$,
  $k\in\{0,\dots,n-1\}$.
  Then for any $\vt\in\R^n$,  
  \begin{equation*}
       \LE(K\cap(\vt+L))\leq 2^{k}\LE(K\cap L),
     \end{equation*}
     and the inequality is best possible.
     \label{lem:brunn}     
\end{lemma}
For various discrete versions of the classcial Brunn-Minkowski theorem, which in
particular implies Brunn's concavity principle we refer to
\cite{dmink-gg, slomka, hernandez, iglesias}.

From \eqref{eq:rev-meyer-inhom} we get for $K\in\Ko^n$ immediately an inequality
of the type
\begin{equation}
  \LE(K)^{\frac{n-1}{n}}\leq c_n  \max_{\vt\in\Z^n,\vu\in\Z^n\setminus\{\vnull\}}\LE(K\cap(\vt+\vu^\perp)).
\label{eq:disslisym}  
\end{equation}
with $c_n=O(n^2)$. 
This may be regarded as a lattice version of the well-known \emph{slicing problem}
for volumes asking for the correct order of a constant $c$ such that
for all  centered convex bodies $K\in\K^n$ there exists a 
$\vu\in\R^n\setminus\{\vnull\}$  
such that
\begin{equation} 
  \vol(K)^\frac{n-1}{n} \leq c\, \vol(K\cap \vu^\perp).
\label{eq:volsli}   
\end{equation} 
To this day, the best known bound $c\leq n^{1/4}$ is due to Klartag
\cite{klartag}. 
With respect to the discrete slicing inequality \eqref{eq:disslisym}
we prove
\begin{theorem} Let $K\in\K^n$. 
    Then 
\begin{equation}
\label{eq:disslyasym}   
 \LE(K)^{\frac{n-1}{n}}\leq O(n^2)
  \max_{\vt\in\Z^n,\vu\in\Z^n\setminus\{\vnull\}}\LE(K\cap(\vt+\vu^\perp)).
\end{equation}
If $K\in\Ko^n$, the constant can be replaced by $O(n)$.
\label{thm:dslicing}
\end{theorem}

Finally, we will give an example that shows that all the constants in \eqref{eq:rev-meyer-hom}, \eqref{eq:rev-meyer-inhom}, \eqref{eq:disslisym} and \eqref{eq:disslyasym} must be at least of order $\sqrt{n}$.

\begin{theorem}
\label{thm:lower-slicing-bound}
For $n\in\N$ there exists a sequence of symmetric convex bodies $(K_j)_{j\in\N}\subseteq\Ko^n$ such that
\begin{equation*}
\limsup_{j\rightarrow\infty} \frac{\LE(K_j)^{\frac{n-1}{n}}}{\sup_H \LE(K_j\cap H)} \geq c\sqrt{n},
\end{equation*}
where $H$ ranges over all affine hyperplanes in $\R^n$ and $c>0$ is a universal constant.
\end{theorem}





We want to remark that the slicing problem \eqref{eq:volsli} has been
extensively studied also for other measures. For instance, Koldobsky
\cite{koldobsky} proved
\begin{equation*}
    \mu(K) \leq  O(\sqrt{n}) \max_{\vx\in\R^n\setminus\{\vnull\}} \mu(K\cap \vx^\perp) \vol(K)^{1/n}
\end{equation*}
for measures $\mu$ that admit a continuous density (see also
\cite{KlartagLivshyts2018} and the
references within).
Chasapis, Giannopoulos and Liakopoulos \cite{cgl} extended this result
to lower dimensional sections  of not necessarily symmetric convex
bodies
\begin{equation}
\label{eq:cgl}
\mu(K) \leq  O(k)^{(n-k)/2} \max_F \mu(K\cap F) \vol(K)^{\frac{n-k}{n}},
\end{equation}
where $F$ ranges over all $k$--dimensional subspaces of $\R^n$ and  
$\mu$ is a measure with a locally integrable density function.
 In \cite{ahz} the authors obtained an inequality similar to
 \eqref{eq:cgl} for $K\in\Ko^n$ and the lattice point enumerator 
 \begin{equation*}
   \LE(K) \leq O(n)^{n-k}  \max_H \LE(K\cap F) \vol(K)^{\frac{n-k}{n}}, 
 \end{equation*}
where $F$ ranges over all $k$--dimensional linear subspaces with $\dim
(F\cap\Z^n)=k$. 
 In the case $k=n-1$ and
 convex bodies of ``small'' volume,  Regev \cite{regev} proved via a
 probablistic approach such an
 inequality with the constant $O(n)$ instead of $O(n)^{n-1}$. 

Finally, we discuss a
reverse Loomis-Whitney inequality in the sprit of
\eqref{eq:rev_meyer}, 
Campi, Gritzmann, and Gronchi \cite{CampiGritzmannGronchi2018} showed that there exists a
constant $\tilde d_n\geq c/n$, where $c$ is an absolute constant,  such that  
\begin{equation}
\label{eq:rev_lw}
\vol(K)^{\frac{n-1}{n}} \geq \tilde d_n \left(\prod_{i=1}^n \vol(K|\vu_i^\perp)\right)^{\frac{1}{n}},
\end{equation} 
where again $\vu_1,\dots,\vu_n$ form an suitable orthonormal basis. In
\cite{rev_lw}, Koldobsky, Saroglou and Zvavitch showed that the
optimal order of the constant $\tilde d_n$ is of size $n^{-1/2}$.

In order to get a  meaningful discrete version of \eqref{eq:rev_lw}  we
have to project so that $\Z^n|\vu_i^\perp$ is again a lattice, i.e.,
$\vu_i\in\Z^n$,  and we have  
to count the lattice points of $K|\vu_i^\perp$ with respect to this
lattice.

\begin{theorem}
\label{thm:drev_lw}
Let $K\in\Ko^n$ with $\dim(K\cap\Z^n)=n$. There exist linearly
independent vectors $\vv_1,...,\vv_n\in\Z^n$ such that 
\begin{equation*} 
\LE(K)^{\frac{n-1}{n}} \geq \Omega(1)^n\Big(\prod_{i=1}^n
\LE_{\Z^n|\vv_i^\perp}(K|\vv_i^\perp)\Big)^{\frac{1}{n}},
\end{equation*}
where $\LE_{\Z^n|\vv_i^\perp}(K|\vv_i^\perp)=|(K|\vv_i^\perp)\cap ( \Z^n|\vv_i^\perp)|$.
\end{theorem}

 The paper is organized as follows. In the next section we recall
 briefly some basic definitions and tools from Convex Geometry and
 Geometry of Numbers needed for  the proofs. Section 3 is devoted to
 the  slicing inequalities, in particular, we provide the
 proofs of the Theorems  \ref{thm:dmeyer} to  \ref{thm:lower-slicing-bound}
  as well as of Lemma \ref{lem:brunn}. The proof of
 Theorem \ref{thm:drev_lw} is given in Section 4, and in the final
 section we discuss improvements for the special class of unconditional bodies.



\section{Preliminaries}
\label{sec:preliminaries}

For a non-zero vector $\vx\in\R^n$, we denote its orthogonal
complement by $\vx^\perp$ and its Euclidean norm by $|\vx|$. We write $B_n=\{\vx\in\R^n:|\vx|\leq 1\}$ for the Euclidean unit ball and $C_n=[-1,1]^n$ for the symmetric cube. The \emph{interval} between $\vx,\vy\in\R^n$ is defined as $[\vx,\vy]=\{\lambda\vx+(1-\lambda)\vy:\lambda\in[0,1]\}$. For two sets $A,B\subseteq\R^n$ the
\emph{Minkowski addition} is defined elementwise, i.e.,
$A+B=\{\va+\vb:\va\in A, \vb\in B\}$. Similarly, for a scalar $\lambda\in\R$, one defines $\lambda A = \{\lambda\, \va: \va\in A\}$ and we write $-A=(-1)A$. 

A \emph{convex body} is a compact convex set $K\subseteq\R^n$. We say
that $K$ is \emph{origin symmetric}, if $K=-K$. The set of all convex
bodies is denoted by $\K^n$ and the set of all origin symmetric convex bodies
is denoted by $\K_o^n$. The \emph{support function} of a convex body
$K$ is defined for $\vx\in\R^n$ as $\suk_K(\vx)=\sup_{\vy\in K}
\langle \vx,\vy\rangle$. If the origin is an interior point of $K$,
the \emph{polar body} of $K$ is defined as
\begin{equation*}
K^\star = \{\vy \in\R^n: \langle \vx,\vy\rangle\leq 1, \forall \vx\in
K\}\in\K^n.
\end{equation*}
Moreover, for such $K$, the \emph{gauge function}
$|\cdot|_K:\R^n\to\R_{\geq 0}$ is defined by $|\vx|_K=\min\{\mu\geq
0: \vx\in\mu K\}$;  it is  $|\cdot|_{K^\star}=\suk_K(\cdot)$
\cite[Lemma 1.7.13]{schneider}. The \emph{volume} $\vol(K)$ of a convex body $K$ is its $n$-dimensional Lebesgue measure. If $K$ is contained in an $k$-dimensional space $F$, we denote by $\vol_k(K)$ its $k$-dimensional Lebesgue measure in $F$.  It is a famous open problem in Convex Geometry
to find the best possible lower bound on the volume product
$\vol(K)\vol(K^\star)$, where $K\in\Ko^n$. Mahler conjectured that it is $4^n/n!$ and it
is known to be true with  $\pi^n/n!$ \cite{kuperberg}. Here we will just use
\begin{equation}
         \vol(K)\vol(K^\star)\geq \frac{3^n}{n!}.
  \label{eq:mahler}
\end{equation}
For  $X\subseteq\R^n$, we denote the \emph{convex hull} of $X$ by $\conv(X)$. If $X$ is finite, $\conv(X)$ is called a \emph{polytope} and if, in addition, $X\subseteq\Z^n$, we call $\conv(X)$ a \emph{lattice polytope}. 

In general, a \emph{lattice} $\Lambda\subseteq\R^n$ is a discrete subgroup of
$\R^n$ of the form
\begin{equation*} 
  \Lambda = \Big\{\sum_{i=1}^k \alpha_i \vb_i: \alpha_i\in\Z\Big\},
\end{equation*} 
for some linearly independent $\vb_1,...,\vb_k\in\R^n$.  The set
$\{\vb_1,...,\vb_k\}$ is called a \emph{(lattice) basis} of
$\Lambda$ and one defines
$\det(\Lambda)=\vol_k([\vnull,\vb_1]+...+[\vnull,\vb_k])$. Sublattices of $\Lambda$ that arise as intersections $\Lambda\cap L$, where $L\subseteq\R^n$ is a linear subspace, are called \emph{primitive}. If $L$ fullfills $\dim L = \dim(\Lambda\cap L)$, $L$ is called a \emph{lattice subspace} of $\Lambda$. A point $\vv\in\Lambda\setminus\{\vnull\}$ is called primitive, if $\Z \vv$ is a primitive sublattice of $\Lambda$. The \emph{polar lattice} of $\Lambda$ is defined as 
\begin{equation*}
\Lambda^\star = \{\va\in\R^n: \langle \vb,\va\rangle \in\Z,\forall \vb\in\Lambda\}. 
\end{equation*} 
There are several duality relations between $\Lambda$ and
$\Lambda^\star$ of which we recall a few here
(cf.~e.g.~\cite[Proposition 1.3.4]{Martinet2003}). First of all the
determinants of $\Lambda$ and $\Lambda^\star$ are linked by the simple
formula $\det(\Lambda)\det(\Lambda^\star) = 1$. Further, a
$k$-dimensional subspace $L\subseteq\R^n$ is a lattice subspace of
$\Lambda$, if and only if $L^\perp$ is an $(n-k)$-dimensional lattice
subspace of $\Lambda^\star$. In particular, every lattice hyperplane
$H$ of $\Lambda$ possesses a primitive normal vector
$\vv^\star\in\Lambda^\star$ and the determinant of $\Lambda\cap H$ is
given by $|\vv^\star|\det\Lambda$. Moreover, the orthogonal projection $\Lambda^\star|L$ is a $k$-dimensional lattice and we have the following relation:
\begin{equation}
\label{eq:polarity1}
(\Lambda\cap L)^\star = \Lambda^\star|L, 
\end{equation}
where on the left-hand side, the polarity operation is taken within the space $L$.

For a set $A\subset\R^n$ we denote by $|A|$ its cardinality and 
for the \emph{lattice point enumerator} of a set $A\subseteq\R^n$ with
respect to a lattice $\Lambda\subseteq\R^n$ we write $\LE_\Lambda
A=|A\cap\Lambda|$. If $\Lambda=\Z^n$, we just write $\LE A$ instead of
$\LE_{\Z^n} A$.

Minkowski established via his successive minima  various fundamental
relations between the volume and lattice point properties of a
symmetric convex body. For $K\in\K_o^n$ and a lattice $\Lambda\subseteq\R^n$, both
full--dimensional, the $i$th \emph{successive minimum} is defined as
\begin{equation*} 
\lambda_i(K,\Lambda)=\min\big\{\lambda>0:\dim(\lambda K \cap \Lambda)
= i\big\},
\end{equation*} 
where $1\leq i \leq n$; we abbreviate $\lambda_i(K) = \lambda_i(K,\Z^n)$.
 Among other Minkowski proved \cite[Ch.VIII, Theorem V]{cassels}
\begin{equation}
  \label{eq:minkowski}
  \frac{2^n}{n!}\det(\Lambda)\leq
  \lambda_1(K,\Lambda)\cdot\ldots\cdot\lambda_n(K,\Lambda)\vol(K)\leq 2^n\det(\Lambda).
\end{equation}
Here we also need a discrete variant of the upper bound going back to Betke et al.~\cite{bhw}. For $K\in\Ko^n$ they proved  
\begin{equation*}\label{eq:lower-bhw}
\LE(K)\leq \prod_{i=1}^n \Big(\frac{2i}{\lambda_i(K)}+1\Big),
\end{equation*}
which was later improved in
\cite{Henk2002}. The currently best known upper bound is due to Malikiosis \cite{Malikiosis2010},
which, in particular, implies for $K\in\Ko^n$ 
\begin{equation}\label{eq:upper-bhw}
  \LE(K)\leq \sqrt{3}^{n-1}\prod_{i=1}^n  \Big\lfloor\frac{2}{\lambda_i(K)}+1\Big\rfloor.
\end{equation}
In general, linearly independent lattice points $\va_i\in\Lambda$, $1\leq
i\leq n$, corresponding to the successive minima, i.e.,
$\va_i\in\lambda_i(K,\Lambda)\,K$, do not form a basis of $\Lambda$. It was shown
by Mahler (cf.~\cite[Sec. 2.10]{geometryofnumbers}), however, that
there exists 
a lattice basis $\vb_1,...,\vb_n\in\Lambda$ such that
\begin{equation} 
  |\vb_i|_K \leq \max\left\{1,\frac{i}{2}\right\} \lambda_i(K,\Lambda).
\label{eq:thm:basis}  
\end{equation}
Next we also need a lower bound on the product of the successive minima
of a convex body $K$ with the ones of the polar body $K^\star$ (cf.~\cite[Theorem 23.2]{Gruber2007}):
\begin{equation}
  \label{eq:5}
  \lambda_i(K^\star,\Lambda^\star)\lambda_{n+1-i}(K,\Lambda)\geq 1. 
\end{equation}

Regarding upper and lower bounds on the volume in terms of the lattice
point enumerator we mention here two results. First, van der Corput
\cite[Ch.2, Theorem 7.1]{geometryofnumbers} proved for $K\in\Ko^n$
\begin{equation}
\label{eq:vdcorput}
\vol(K)\leq \big(2^{n-1} (\LE_\Lambda(K)+1)\big)\det\Lambda,
\end{equation}
and Blichfeldt \cite{Blichfeldt1921} showed for $K\in\K^n$ with 
$\dim(K\cap\Lambda)=n$
\begin{equation}
\label{eq:blichfeldt}
\vol(K)\geq \frac{1}{n!}\Big(\LE_\Lambda(K)-n\Big)\det\Lambda.
\end{equation}
Finally, the volume and the lattice point enumerator are equivalent ``on a large scale'', i.e., for any $n$-dimensional convex body $K\subseteq\R^n$, $n$-dimensional lattice $\Lambda\subseteq\R^n$ and $\vt\in\R^n$ one has (cf.\ e.g.\ \cite[Lemma 3.22]{ad-comb})
\begin{equation}
\label{eq:vol-appr}
\lim_{r\rightarrow\infty} \frac{\LE_{\vt+\Lambda}(rK)}{r^n}=\frac{\vol(K)}{\det\Lambda}
\end{equation}

In this paper we will mostly deal with the standard
lattice 
$\Z^n$ since all the results can easily be generalized to arbitrary
lattices.

For more information on Geometry of Numbers and/or Convex Geometry we
refer to the books \cite{ Gardner2006, Gruber2007, geometryofnumbers, schneider}.

\section{Slicing inequalities for the lattice point enumerator}
\label{sec:slicing} 

First, we show that the answer to the question
\eqref{eq:dmeyer-question} of Gardner et al.~is in general negative, if the dimension is greater than
2. 
\begin{proposition}
\label{prop:simplex}
Let $n\geq 3$ be fixed.  There exists no positive number $c>0$ such
that for all $K\in\K^n$.
\begin{equation}
\label{eq:noc}
\LE(K)^{\frac{n-1}{n}} \geq c \Big(\prod_{i=1}^n \LE(K\cap\ve_i^\perp)\Big)^\frac{1}{n}.
\end{equation}
\end{proposition}

\begin{proof}
 We first prove it for $n=3$.  For an integer $r\in\N$, let $T_k$ be the simplex with
 vertices $\{\vnull, \ve_1, \ve_1+k\,\ve_2, k\,\ve_3\}$ (see Figure
 \ref{fig:simplex}).
\begin{figure}[htb] 
\centering
\includegraphics[width = .3\textwidth]{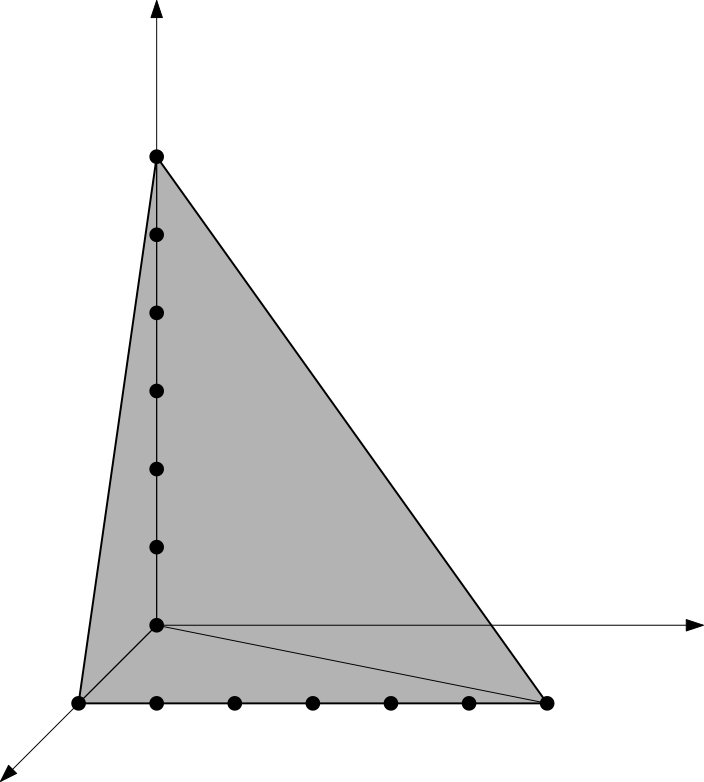}
\caption{The simplex $T_r$.} 
\label{fig:simplex}
\end{figure}
 Then, 
$\LE(T_k) = 2(k+1)$ and also $\LE(T_k\cap \ve_1^\perp) = k+1$ and
$\LE(T_k\cap \ve_2^\perp)=\LE(T_k\cap \ve_3^\perp)=k+2$.  Thus
\begin{equation*}
       \frac{\LE(T_k)^{\frac{2}{3}}}{\Big(\prod_{i=1}^3
         \LE(T_k\cap\ve_i^\perp)\Big)^\frac{1}{3}} \leq
       2^{\frac{2}{3}}\frac{(k+1)^{\frac{2}{3}}}{k+1} = 2^{\frac{2}{3}}(k+1)^{\frac{-1}{3}}
     \end{equation*}
     and so for $k\to\infty$ the left hand sides tends to $0$.

For $n\geq 4$ we just can consider, e.g.,  the simplices $\conv(T_k\cup\{\ve_4,\dots,\ve_n\})$.     
\end{proof}

Roughly speaking, the simplex $T_k$ from above falsifies
\eqref{eq:noc} because the two skew segments $[\vnull,k\ve_3]$ and
$[\ve_1, \ve_1+k\,\ve_2]$ are both "long", but do not generate any
additional points in $T_k$. Such a construction is not possible in the
symmetric case. In fact, if $K\in\Ko^n$ possesses $2h+1$ points on the
coordinate axis $\R\ve_n$, any interior point $\vv\in K\cap\ve_n^\perp$ will
contribute $O_{v,n}(h)$ points to $K$. Here, $O_{v,n}$ hides a constant that only depends on $v$ and $n$.
However, unlike the simplex above, a \emph{symmetric} convex body always contains at least $\LE(K)/3^n$--many interior lattice points (see \cite{merino-schymura}). Motivated by this heuristic, we conjecture the following polytopes to be extremal in \eqref{eq:noc}, when restricted to $\Ko^n$.

\begin{example}
\label{ex:berg} 
For an integer $h\in\N$ let
$K_h=\conv\big((C_{n-1}\times\{\vnull\})\cup\{\pm h\,\ve_n\}\big)$ be
a  double pyramid over the $(n-1)$-dimensional cube
$C_{n-1}=[-1,1]^{n-1}$ (see Figure \ref{fig:berg}).
\begin{figure}[hbt]
\centering
\includegraphics[width=.3\textwidth]{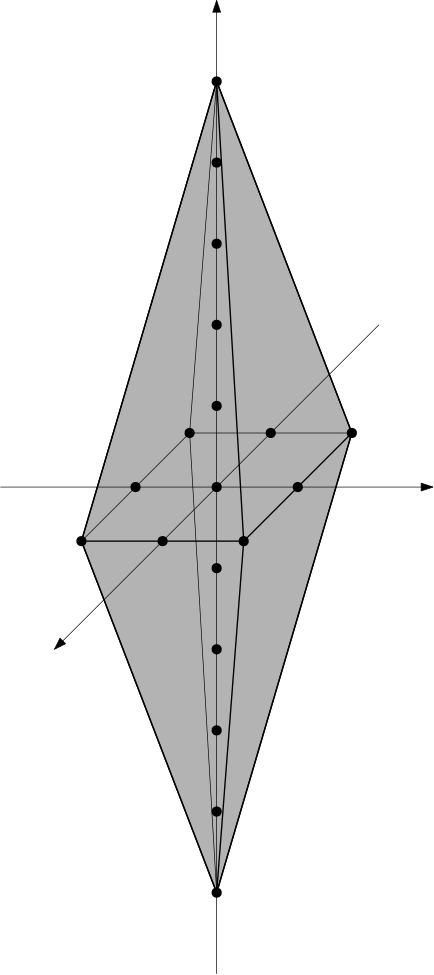}
\caption{The double pyramid $K_h$.}
\label{fig:berg}
\end{figure}
Then $\LE(K_h) =
3^{n-1} +2h$, $\LE(K_h\cap \ve_i^\perp) = 3^{n-2}+2h$, for $1\leq
i<n$, and $\LE(K_h\cap \ve_n^\perp) = 3^{n-1}$. Thus, 
\begin{equation*} 
\lim_{h\rightarrow\infty}  \frac{\LE(K_h)^{n-1}}{\prod_{i=1}^n \LE(K_h\cap
  \ve_i^\perp)} = \frac{1}{3^{n-1}},
\end{equation*}
which is why we conjecture the optimal constant in \eqref{eq:noc} to be $3^{-\frac{n-1}{n}} \approx 1/3$.
\end{example}

In order to prove the lower bound in Theorem \ref{thm:dmeyer}, we would first like to
understand the behaviour of $\LE(\cdot)$ with respect to affine
transformations. An important tool for this is the \emph{index} of a
sublattice $\Lambda^\prime\subseteq\Lambda$, i.e., the number of
different cosets $\va+\Lambda^\prime$, $\va\in\Lambda$. If
$\dim\Lambda^\prime=\dim\Lambda$, this number is known to be finite
and is given by  $\det(\Lambda^\prime)/\det(\Lambda)$. The following lemma will be used extensively throughout this paper.

\begin{lemma}
\label{lemma:elementary-trafos}
Let $K\in\K_o^n$, $A\in\Z^{n,n}$ be a regular matrix and
$\vt\in\R^n$. Then,
\begin{equation*} 
\LE\big( A\,K+\vt\big) \leq 2^{n-1} |\det A| (\LE(K)+1)\leq 2^{n} |\det
A| \LE(K).
\end{equation*} 
In particular, we have
\begin{equation}
\label{eq:translations}
\LE(K+\vt)\leq 2^{n-1} (\LE(K)+1)\leq 2^n\LE(K),
\end{equation}
which is best possible, and  for $m\in\N$,
\begin{equation}
\label{eq:dilates}
\LE(m\,K)\leq 2^{n-1}\,m^n(\LE(K)+1)\leq(2m)^n\, \LE(K).
\end{equation}
which is best possible up to a factor
$(1+\frac{1}{2m-1})^n$.
\end{lemma}

\begin{proof}
Let $\Lambda=2A\Z^n\subseteq\Z^n$ and let $\Gamma_i\subseteq\Z^n$,
$1\leq i\leq 2^n|\det A|$, be the cosets of $\Lambda$ in
$\Z^n$. Consider two points $\vy_j=A\,\vx_j+\vt\in A\,K+\vt$, where
$\vx_j\in K$, $j=1,2$, that belong to a common $\Gamma_i$, say. For such
points, we have $\vy_1-\vy_2\in\Lambda$. Thus, by the symmetry of $K$,
we have
\begin{equation*} 
\frac{1}{2}(\vy_1-\vy_2)=A\Big(\frac{1}{2}(\vx_1-\vx_2)\Big)\in AK\cap A\Z^n =
A(K\cap\Z^n).
\end{equation*} 
That means
\begin{equation} 
\big|(A\,K+\vt)\cap\Gamma_i - (A\,K+\vt)\cap\Gamma_i\big|\leq \big|A(K\cap\Z^n)\big| =
\LE(K).
\label{eq:1}
\end{equation}
Since for any two finite sets $A,B\subseteq\R^n$  (cf.~\cite[Section 5.1]{ad-comb})
\begin{equation}
\label{eq:finite-sets}
|A+B|\geq |A|+|B|-1,
\end{equation}
we get from \eqref{eq:1} $\big|(A\,K+\vt)\cap\Gamma_i\big|\leq (\LE(K)+1)/2$. Since the $\Gamma_i$'s form a partition of $\Z^n$, the desired inequality follows.

In order to see that \eqref{eq:translations} is best-possible,
consider the rectangular box $Q_k=\frac{1}{2}[-1,1]^{n-1}\times
[-k+\frac{1}{2}, k-\frac{1}{2}]$, where $k\in\N$, and
$\vt=(\frac{1}{2},...,\frac{1}{2})^T\in\R^n$. Then, we have $\LE(Q_k)= 2k-1$ and $\LE(Q_k+\vt) = 2^{n-1}2k =2^{n-1}(\LE(Q_k)+1).$

For \eqref{eq:dilates}, let $K=\big[ -(1-\frac{1}{2m}),
1-\frac{1}{2m}\big]^n$; then $\LE(m\,K) = (2m-1)^n \LE(K)$.
\end{proof}

\begin{remark} \hfill
  \begin{enumerate}
\item Restricted to the class of origin symmetric lattice polytopes,
\eqref{eq:translations}
is best-possible up to a factor 4, as the polytopes of Example
\ref{ex:berg} together with the vector $\vt=(1/2,...,1/2,0)^T$
show: On the one hand, one has $\LE(K_h)=2h+O(1)$, where
$O(\cdot)$ describes the asymptotic behaviour for
$h\rightarrow\infty$.
On the other hand, the cube
$[0,1]^{n-1}$ is contained in the relative interior of $K\cap
\ve_n^\perp + \vt$. Even more, each of its vertices $\vv$ is the
midpoint of $[\vt,\vt+\vx]$, where $\vx$ is a vertex of
$K\cap\ve_n^\perp$. Hence, the lines $\vv+\R\ve_n$ contribute $h+O(1)$
points to $(K_h+\vt)\cap\Z^n$ each. Since there are $2^{n-1}$ such lines, we obtain
\begin{equation*}
\lim_{h\rightarrow\infty} \frac{\LE(K_h+\vt)}{\LE K_h} = \lim_{h\rightarrow\infty}\frac{2^{n-1} (h+O(1))}{2h+O(1)} = 2^{n-2}.
\end{equation*}

In fact, Wills \cite{wills}  showed that for any lattice polygon
$P\subseteq\R^2$ and $\vt\in\R^2$ one has $\LE(P+\vt)\leq \LE(P)$. We
conjecture that for any lattice polytope $P\in\Ko^n$ 
\begin{equation*} 
  \LE(P+\vt) \leq 2^{n-2} \LE(P).
\end{equation*} 
\item
If $K$ is not necessarily symmetric,
\eqref{eq:translations} and \eqref{eq:dilates} fail. In that case,
counterexamples are given by the simplices $T_k$ in the proof of
Proposition \ref{prop:simplex}. Basically, the reason for this is that
$T_k$ contains $O(k)$ lattice points, while in a translation or
dilation of $T_k$ one may find a rectangular triangle spanned by two
orthogonal segments of length $O(k)$ each, lying in a hyperplane of
the form $\ve_i^\perp+\vt$, $\vt\in\Z^3$. Such a triangle contributes
$O(k^2)$ points. Again, letting $k\rightarrow\infty$ shows that the
inequalities cannot be generalized to the non-symmetric case. 
This has also been observed
independently by Lovett and Regev in \cite{lovett-regev}.
\end{enumerate}
\end{remark}

Next we come to the proof of Theorem \ref{thm:dmeyer}. 
\begin{proof}[Proof of Theorem \ref{thm:dmeyer}]
  Let $K_i=K\cap \ve_i^\perp$ and
  $\Lambda_i=\Z^n\cap\ve_i^\perp$, $1\leq i\leq n$,. First, assume that there is an $i$ such that $\frac{1}{2}K_i$ contains only the origin as a lattice point. For convenience, let $i=n$. Then Lemma \ref{lemma:elementary-trafos}, \eqref{eq:dilates}, yields $\LE_{\Lambda_n} K_n\leq 4^{n-1}\LE_{\Lambda_n}((1/2)K_n)=4^{n-1}$. Thus, since $K_i\subseteq K$, 
  \begin{equation*}
  \LE(K)^{n-1}\geq \prod_{i=1}^{n-1} \LE_{\Lambda_i}K_i \geq \frac{1}{4^{n-1}}\prod_{i=1}^n \LE_{\Lambda_i} K_i > \frac{1}{4^{n(n-1)}}\prod_{i=1}^n \LE_{\Lambda_i} K_i,
  \end{equation*}
  and we are done. So we can assume that every $\frac{1}{2}K_i$ contains a non-zero lattice point. In this case, the second inequality in Lemma \ref{lemma:elementary-trafos}, \eqref{eq:dilates}, is strict and we obtain
%
  \begin{equation}
	\LE K_i = \LE_{\Lambda_i}K_i=\LE_{\Lambda_i}\Big(2\frac{1}{2}K_i\Big)< 4^{n-1}\LE_{\Lambda_i}\Big(\frac{1}{2}K_i\Big) = 4^{n-1}\LE_{2\Lambda_i}(K_i).
        \label{eq:ineq2}
  \end{equation}   
  Now we consider the linear map
  \begin{equation*} 
    \begin{split} 
     \phi &: (K_1\cap 2\Lambda_1)\times \cdots\times (K_n\cap
     2\Lambda_n)\to (K\cap\Z^n)^{n-1} \text{ given by } \\ 
     &\phi\big( (\va_1,\dots,\va_n) \big) = \big(\frac{1}{2}(\va_2-\va_1), \frac{1}{2}(\va_3-\va_1),\dots, \frac{1}{2}(\va_n-\va_1)\big).
     \end{split}
   \end{equation*}
  We note that by the symmetry and convexity of $K$ as well as the definition of
  $2\,\Lambda_i$ we readily have $\phi\big(
  (\va_1,\dots,\va_n)\big)\in (K\cap\Z^n)^{n-1}$. Moreover, 
  since $\va_i\in\ve_i^\perp$, $1\leq i\leq n$,  the map is injective
  and so
  \begin{equation*}
    \LE K^{n-1}\geq\prod_{i=1}^n\LE_{2\Lambda_i}K_i.
  \end{equation*}
  Together with \eqref{eq:ineq2} we obtain
  \begin{equation*}
    \LE(K)^{\frac{n-1}{n}} > \frac{1}{4^{n-1}}\big(\prod_{i=1}^n\LE(K_i)\big)^{\frac{1}{n}}.
  \end{equation*}

\end{proof}

%
%

Next, we want to reverse the inequality in Theorem \ref{thm:dmeyer}.
Apparently, one cannot get  an upper bound on $\LE(K)$ in terms of the
geometric mean of the sections $K\cap\ve_i^\perp$. Here we have to
replace the $\ve_i$ by a lattice basis $\vb_1,...,\vb_n$ of $\Z^n$  that "suits" the body $K$. Our strategy will then be to decompose $\Z^n=\cup_{j\in\Z}\{\vx\in\Z^n:\langle \vx,\vb_i\rangle = j\}$ and estimate the sections parallel to $\vb_i^\perp$ against the central one.  As for the volume \emph{Brunn's concavity principle} states that
\begin{equation}
\label{eq:brunn}
\vol\big( K\cap (\vt+L) \big) \leq \vol(K\cap L),
\end{equation}
for any $k$-dimensional linear subspace $L\subseteq\R^n$, $\vt\in\R^n$ and $K\in\Ko^n$. So the volume-maximal section of $K$ parallel
to $L$ is indeed always the one containing the origin. Unfortunately, this is
false in the discrete setting as the following example shows:
Let $K =\conv\big(\pm([0,1]^{n-1}\times\{1\})\big)$ and for $1\leq
k\leq n-1$, let  $L_k=\mathrm{span}\{\ve_1,...,\ve_k\}$. Then
$\#(K\cap L_k) = 1$, but $\#(K\cap(\ve_n+L_k)) = 2^k$.

Indeed, the deviation between the central section and the maximal
section is extremal for $K$ as above, as claimed in Lemma
\ref{lem:brunn} which we prove next.
\begin{proof}[Proof of Lemma \ref{lem:brunn}]
We may assume  $\vt\in\Z^n$. In that case, $\Lambda = L\cap\Z^n$ and
$\Lambda^\prime=\Lambda+\vt$ are $k$-dimensional (affine)
sublattices 
of $\Z^n$ and both of them intersect exactly $2^k$ cosets of
$\Z^n/2\Z^n$: Let $\Gamma\in \Z^n/2\Z^n$ such that
$\Lambda\cap \Gamma\ne\emptyset$. Then
$\Lambda\cap\Gamma\in\Lambda/2\Lambda$ and  since two different cosest
do not have a common point, we see that $\Lambda$ intersects at most
$|\Lambda/2\Lambda|=2^k$ cosets of $\Z^n/2\Z^n$. Conversely, since
$\Lambda$ arises as a section of $\Z^n$ with a linear subspace, there exists 
a complementary lattice $\overline{\Lambda}\subseteq\Z^n$ such that
$\Lambda\oplus\overline{\Lambda}=\Z^n$. Then, every coset
$\tilde \Gamma\in\Lambda/2\Lambda$ defines a unique coset
$\tilde\Gamma\oplus 2\overline{\Lambda}$ of $\Z^n/2\Z^n$  and so
$\Lambda$ meets at least $2^k$ cosets of $\Z^n/2\Z^n$. Regarding the
affine lattice $\Lambda^\prime$, we note that the translation by $\vt$
serves as a bijection between cosets in $\Lambda$ and the cosets in $\Lambda^\prime$.  

Now consider two points $\vx,\vy\in K\cap \Lambda^\prime$ belonging to
a common coset $\Gamma$ of $ \Z^n/2\Z^n$. By the symmetry of $K$, we
have $\frac{1}{2}(\vx-\vy)\in K\cap \Lambda$ and so
\begin{equation*} 
  \big|K\cap\Lambda'\cap\Gamma-K\cap\Lambda'\cap\Gamma\big| \leq |K\cap\Lambda|.
\end{equation*}   
Thus, by \eqref{eq:finite-sets}, every coset of $\Z^n/2\Z^n$ that is
present in $\Lambda^\prime\cap K$ contains at most
$\frac{1}{2}(\#(K\cap \Lambda) +1)$ points of $K\cap \Lambda'$. Hence,
\begin{equation*} 
   |K\cap\Lambda'|\leq 2^k\, \frac{1}{2}( |K\cap \Lambda| +1).
\end{equation*}   
\end{proof}

Now we are ready for the proof of the reverse Meyer Theorem \ref{thm:drev_meyer}. 
\begin{proof}[Proof of Theorem \ref{thm:drev_meyer}] 
By induction on the dimension, we will show that for any $n$-dimensional convex body $K\in\Ko^n$ and any $n$-dimensional lattice $\Lambda$, there exists a basis $\vb_1,...,\vb_n$ of $\Lambda^\star$ and vectors $\vt_1,...,\vt_n\in\Lambda$ such that 
\begin{equation}
\label{eq:induction-hyp}
\LE_\Lambda(K)^{n-1} \leq (n!)^2 4^n \prod_{i=1}^n \LE_\Lambda\big(K\cap(\vt_i+\vb_i^\perp)\big).
\end{equation} 
From this, \eqref{eq:rev-meyer-inhom} follows by considering
$\Lambda=\Z^n$ and taking the $n$th root. Moreover, \eqref{eq:rev-meyer-hom} follows immediately from \eqref{eq:rev-meyer-inhom} and Lemma \ref{lem:brunn}.
  
  First, we assume $\dim(K\cap\Lambda)=n$.
  For any $\vb\in\Lambda^\star$ we may write
  \begin{equation}
    \begin{split} 
    \LE_\Lambda K &= \sum_{i=-\lfloor\suk(K,\vb)\rfloor}^{\lfloor\suk(K,\vb)\rfloor}
    \LE_\Lambda\big(K \cap \{\vx\in\R^n :\langle \vb,\vx\rangle = i\}\big) \\ 
        &\leq (2\lfloor \suk(K,\vb)\rfloor+1) \LE_\Lambda\big(K\cap(\vt_\vb+\vb^\perp)),
      \end{split}
  \label{eq:pigeon-slicing}    
  \end{equation}
  where $\vt_\vb\in\Lambda$ is chosen to be the translation that maximizes
  the number of lattice points in a section parallel to
  $\vb^\perp$.

  Now let $\vb_1,\dots,\vb_n\in\Lambda^\star$ be a basis of $\Lambda^\star$ obtained
  from \eqref{eq:thm:basis} with respect to the polar body $K^\star$,
  i.e., we have $|\vb_i|_{K^\star}\leq i \lambda_i(K^\star,\Lambda^\star)$, $1\leq
  i\leq n$.  For the vectors $\vb_i$ we denote the above translation vectors $\vt_{\vb_i}$ by
  $\vt_i$. Then, on account of $\suk(K,\vb_i)=|\vb_i|_{K^\star}$ we
  conclude from \eqref{eq:pigeon-slicing}
  \begin{equation} 
    \begin{split} 
    \LE_\Lambda K^n & \leq n!\prod_{i=1}^n \big(2\lambda_i(K^\star,\Lambda^\star)+1\big)
    \prod_{i=1}^n \LE_\Lambda\big(K\cap(\vt_i+\vb_i^\perp)\big)\\ 
             & \leq n! 3^n \prod_{i=1}^n \lambda_i(K^\star,\Lambda^\star)   \prod_{i=1}^n \LE_\Lambda\big(K\cap(\vt_i+\vb_i^\perp)\big),
    \label{eq:later}
  \end{split}  
  \end{equation}   
  where for the last inequality we used $\lambda_i(K^\star,\Lambda^\star)\geq 1$
  which follows from the assumption $\dim(K\cap\Lambda)=n$ via \eqref{eq:5}.

  Using the upper bound of Minkowski's theorem
 \eqref{eq:minkowski}, the lower bound on the volume product \eqref{eq:mahler}
  and van der Corput's inequality \eqref{eq:vdcorput}, we estimate
 
 \begin{equation}
 \label{eq:lambdastars}
 \prod_{i=1}^n \lambda_i(K^\star,\Lambda^\star) \leq \frac{2^n\det\,\Lambda^\star}{\vol(K^\star)}\leq n!\Big(\frac{2}{3}\Big)^n \vol(K)\det\Lambda^\star\leq n!\Big(\frac{4}{3}\Big)^n \LE_\Lambda(K).  
 \end{equation}
 
 Substituting this into \eqref{eq:later} 
 yields the desired inequality \eqref{eq:induction-hyp} for this case.
 
 It remains to consider the case $\dim(K\cap\Lambda)<n$, so let
 $K\cap\Lambda\subseteq H$  for some $(n-1)$-dimensional lattice subspace $H\subseteq\R^n$.
 Let $\Gamma= \Lambda\cap H$ and next 
 we apply our  induction hypothesis to $\Gamma$ and $K\cap H$. Hence, we find a basis $\vy_1,...,
 \vy_{n-1}$ of $\Gamma^\star$ and vectors $\vt_1,...,\vt_{n-1}\in\Gamma$        such that
%
%
%
 \begin{equation}
 \label{eq:rev-meyer-ind}
 \LE_\Gamma(K)^{n-2}\leq (n-1)!^2 4^{n-1} \prod_{i=1}^{n-1} \LE_\Gamma \big(K\cap(\vy_i^\perp+\vt_i)\big),
 \end{equation}
 which is equivalent to
 \begin{equation}
 \label{eq:weird-vectors}
 \LE_\Gamma(K)^{n-1}\leq (n-1)!^2 4^{n-1} \LE_\Gamma(K\cap \vb_n^\perp) \prod_{i=1}^{n-1} \LE_\Gamma\big(K\cap(\vy_i^\perp+\vt_i)\big),
 \end{equation}
 where $\vb_n\in\Lambda^\star$ is a primitive normal vector of $H$. Unfortunately,
 the independent system
 $\{\vy_1,...,\vy_{n-1},\vb_n\}$ is in general not a basis of
 $\Lambda^\star$. In fact, the $\vy_i$'s are not elements of $\Lambda^\star$ in the first   place. 
 
 In view of \eqref{eq:polarity1}, we have $\Gamma^\star = \Lambda^\star|H=\Lambda^\star|\vb_n^\perp$. 
 So there are vectors $\vb_i\in (\vy_i+\R\vb_n)\cap\Lambda^\star$. 
 For these vectors, one has 
 $\vb_i^\perp\cap H = \vy_i^\perp\cap H$. By our assumption on $K$,
 this means
 \begin{equation}
 \label{eq:gammalambda}
 \LE_\Gamma\big(K\cap(\vy_i^\perp+\vt_i)\big) = \LE_\Lambda\big(K\cap(\vb_i^\perp+\vt_i)\big),
 \end{equation}
 for all $1\leq i\leq n-1$. Moreover, $\{\vb_1,...,\vb_n\}$ is a $\Lambda^\star$-basis, since 
 \begin{equation}
 \label{eq:determinants}
 \begin{aligned}
 \det(\vb_1,...,\vb_n) &= \det(\vy_1,...,\vy_{n-1}, \vb_n)\\
 &= |\vb_n|\det\Gamma^\star = \frac{|\vb_n|}{\det\Gamma}=\frac{|\vb_n|}{|\vb_n|\det\Lambda}=\det\Lambda^\star.
 \end{aligned}
 \end{equation}  
 In view of \eqref{eq:weird-vectors} and \eqref{eq:gammalambda}, $\{\vb_1,...,\vb_n\}$ is the desired basis and our proof is complete.

\end{proof}

\begin{remark} We remark that if we just want to find linearly independent lattice
points $\va_i\in\Z^n$, $1\leq i\leq n$, for the slices in Theorem
 \ref{thm:drev_meyer} instead of a basis, then one can save one factor of $n$ in the bounds of Theorem \ref{thm:drev_meyer}. For if, we replace in the proof above the
basis vectors $\vb_i$ by linearly independent lattice points
$\va_i\in\lambda_i(K^\star)K^\star\cap\Z^n$, $1\leq i\leq n$.   In this way
we do not need the estimate \eqref{eq:thm:basis}.
 In particular this leads to
  \begin{equation}
\label{eq:dslicing-sym}
\LE(K)^{\frac{n-1}{n}} < O(n) \max_{\vt\in\Z^n,\vu\in\Z^n\setminus\{\vnull\}} \LE(K\cap (\vt+\vu^\perp)).
\end{equation}
\end{remark} 

In the remainder of this section, we want to generalize \eqref{eq:dslicing-sym} to the non-symmetric case.

If we consider \eqref{eq:pigeon-slicing}, we see that we can also
estimate $\# K$, if $K$ is not symmetric; in that case one replaces
the factor $(2\lfloor \suk(K,\vb)\rfloor+1)$ on the right hand side by the number
of hyperplanes parallel to $\vb^\perp$ that intersect $K$. As it will turn out, the challenge
then is to
compare the number of lattice points in $K-K$ to the number of points
in $K$.

As for the volume, this is known as the \emph{Rogers--Shephard} inequality, which asserts that
\begin{equation}
\label{eq:rogers-shephard}
\vol(K-K) \leq {{2n}\choose{n}} \vol(K).
\end{equation}
The simplices $T_k$ given in the proof of Proposition \ref{prop:simplex} show
that there is no similar inequality for the lattice point enumerator:
While we have $\LE(T_k) = O(k)$, in $T_k-T_k$ we find the triangle $\conv\{0,k/2\ve_2, k\ve_3\}$ which contains $O(k^2)$ lattice points. Since $k$ can be arbitrarily large, there is no constant $c_n>0$ depending only on the dimension such that $\LE(K-K)\leq c_n\LE(K)$.

However, the simplices $T_k$ are extremely flat and therefore easily admit a large hyperplane section. Our strategy in order to prove Theorem \ref{thm:dslicing} will be to argue that convex bodies $K$ whose difference body $K-K$ contains disproportionately many lattice points are automatically flat. This reasoning is inspired by the proof of Theorem 4 in \cite{dhelly}. 

%
%
%

\begin{proof}[Proof of Theorem \ref{thm:dslicing}]
The inequality for $K\in\Ko^n$ is already given by \eqref{eq:dslicing-sym}. So let $K\in\K^n$. We may assume $\dim(K\cap\Z^n) = n$, because otherwise, $K\cap\Z^n$ is contained in a hyperplane itself and the inequality follows directly.
Therefore, $K-K$ contains $n$ linearly independent lattice points and it
follows
$\lambda^\star=\lambda_1\big((K-K)^\star\big)\geq 1$ (cf.~\eqref{eq:5}).  
Let $\vy\in
\lambda^\star(K-K)^\star\cap\Z^n\setminus\{\vnull\}$. Then
\begin{equation*}
    \suk(K,\vy)+\suk(K,-\vy)=\suk(K-K,\vy)=|\vy|_{(K-K)^\star}=\lambda^\star,
\end{equation*}
and similarly to
\eqref{eq:pigeon-slicing}, \eqref{eq:later},  we obtain for a certain $\vt\in\Z^n$
\begin{equation}\label{eq:non-sym-slicing}
\LE K \leq (2\lambda^\star +1) \LE\big(K\cap (\vt+\vy^\perp)\big)\leq 3\lambda^\star \LE\big(K\cap (\vt+\vy^\perp)\big). 
\end{equation}
Let $\vc\in K$ be the centroid of $K$. Then it is known
that \cite[Lemma 2.3.3]{schneider}
\begin{equation}
  \label{eq:3}
                   K-K\subset (n+1)(-\vc+K).
\end{equation}
First we assume that $\frac{1}{2}(K+\vc)=\vc+\frac{1}{2}(-\vc+K)$ does
not contain any integral lattice point. Then the well-known flatness theorem
implies that (cf.~\cite{BanaszczykLitvakPajorEtAl1999})
\begin{equation*} 
\lambda^* =2\lambda_1\left(\big(\frac{1}{2}(\vc+K)-\frac{1}{2}(\vc+K)\big)^\star\right)=O(n^{3/2}).
\end{equation*} 
In view of \eqref{eq:non-sym-slicing} we are done.

So we can assume that there is a lattice point
$\va\in\frac{1}{2}(\vc+K)\cap\Z^n$. By the choice of $\vc$
(cf.~\eqref{eq:3}), we get
\begin{equation}\label{eq:sym-non-sym}
\begin{aligned}
\va+\frac{1}{2(n+1)}(K-K)&\subseteq \frac{1}{2}(\vc+K) + \frac{1}{2(n+1)}(K-K)\\
&=\frac{1}{2}K + \frac{1}{2}\big((\vc+\frac{1}{(n+1)}(K-K)\big)\\
&\subseteq \frac{1}{2}K+\frac{1}{2}K = K.
\end{aligned}
\end{equation}
Now in order to bound $\lambda^*$ in this case we use
\eqref{eq:lambdastars} applied to $(K-K)^\star$ and $\Z^n$, which implies
\begin{equation*} 
 (\lambda^*)^n \leq n!\left(\frac{4}{3}\right)^n\LE(K-K).
\end{equation*}   
Together with  \eqref{eq:non-sym-slicing} we obtain
\begin{equation}
\LE(K)^n \leq n!4^n\LE(K-K)\LE\big(K\cap(\vt+\vy^\perp)\big)^n.
\label{eq:estimate_against_K-K} 
\end{equation}

In order to estimate the number of lattice points in  $K-K$ we may apply Lemma
\ref{lemma:elementary-trafos}, \eqref{eq:dilates} and with
\eqref{eq:sym-non-sym} we get
\begin{equation*}
\LE(K-K)\leq 4^n(n+1)^n \LE\Big(\frac{1}{2(n+1)}(K-K)\Big)\leq
4^n(n+1)^n\LE K.
\end{equation*}
By plugging this into \eqref{eq:estimate_against_K-K}, we obtain
\begin{equation*}
\LE(K)^{n-1} \leq 16^nn!(n+1)^n \LE\big(K\cap(\vt+\vy^\perp)\big)^n.
\end{equation*}
After taking the $n$-th root, we have
$$\LE(K)^{\frac{n-1}{n}}\leq O(n^2)\LE\big(K\cap(\vt+\vy^\perp)\big)^n,$$
as desired.
\end{proof}
We finish the section by proving Theorem \ref{thm:lower-slicing-bound}
providing a lower bound on the constants presented in 
Theorem \ref{thm:drev_meyer} and \ref{thm:dslicing}.
%
%
\begin{proof}[Proof of Theorem \ref{thm:lower-slicing-bound}]
We consider a lattice $\Lambda\subseteq\R^n$ such that $\Lambda$ is
\emph{self-polar}, i.e., $\Lambda = \Lambda^\star$, and $\lambda_1(B_n,\Lambda) = c\sqrt{n}$, where $c$ is an
absolute constant.  Such lattices have been detected by Conway and Thompson
\cite[Theorem 9.5]{bilinear}. 
We will use a volume approximation argument for the Euclidean ball $rB_n=\{\vx\in\R^n:|\vx|\leq r\}$, where $r\rightarrow\infty$.

For $\vx\in\R^n\setminus\{\vnull\}$ and $\alpha\in\R$, let
$H(\vx,\alpha) = \{\vy\in\R^n: \langle \vx,\vy\rangle = \alpha\}$ be
the corresponding hyperplane. For $r>0$ let $\va_r\in\R^n$ and
$\alpha_r\in\R$ such that $\#_\Lambda (rB_n \cap H(\va_r,\alpha_r))$
is maximal. Since $\Lambda$ is self--polar, we may assume that $\va_r\in\Lambda$ and $\alpha\in\Z$. In order to control the limit $r\rightarrow\infty$ we want
to find a sequence of radii $(r_j)_{j\in\N}\subseteq\N$ such that
$r_j\rightarrow\infty$ and $H(\va_{r_j},\alpha_{r_j})$ is constant. To
this end, fix a primitive vector $\va_0\in\Lambda$. Van der Corput's inequality \eqref{eq:vdcorput} yields

\begin{equation}
\label{eq:vdc}
\LE_\Lambda\big(rB_n\cap H(\va_r,\alpha_r)\big) \geq \LE_\Lambda (rB_n\cap \va_0^\perp) \geq 2^{-(n-1)} r^{n-1} \frac{\omega_{n-1}}{|\va_0|}, 
\end{equation}
where $\omega_i$ denotes the volume of the $i$--dimensional Euclidean unit
ball and we used that the determinant of $\Lambda\cap \va_0^\perp$ is
given by $|\va_0|\det\Lambda=|\va_0|$, since the determinant of any self-polar lattice is 1.

On the other hand, if $r$ is large enough, $rB_n$ contains $n$
linearly independent points of $\Lambda$. Thus, the maximal section
$rB_n\cap H(\va_r,\alpha_r)$ contains $(n-1)$ affinely independent
points of $\Lambda$; Otherwise, we might choose another point $\vx\in rB_n\cap\Lambda$
 and replace $H(\va_r,\alpha_r)$ by the affine hull of $rB_n\cap H(\va_r,
 \alpha_r)\cap\Lambda$ and $\vx$. This yields a hyperplane that contains more lattice points of $rB_n$ than $H(\va_r,\alpha_r)$, contradicting the maximality.   Hence, Blichfeldt's inequality
\eqref{eq:blichfeldt} yields
\begin{equation*}
\LE_\Lambda\big(rB_n\cap H(\va_r,\alpha_r) \leq n!\, r^{n-1}
\frac{\omega_{n-1}}{|\va_r|}.
\end{equation*} 
Combining with \eqref{eq:vdc}, we obtain $|\va_r|\leq 2^{n-1}\,n!\,|\va_0|$, for almost all $r\in\N$. Since this bound is independent of $r$, we find a sequence $(r_j)_{j\in\N}\subseteq\N$ that tends to infinity such that $\va_{r_j} = \overline{\va}$, for all $j$ and some primitive $\overline{\va}\in\Lambda$ independent of $j$. 

Since $\overline{\va}\in\Lambda$, we have for any $\alpha>|\overline{\va}|^2$
\begin{equation*} 
-\overline{\va}+\left(r_jB_n\cap H(\overline{\va},
  \alpha)\big)\cap\Lambda\right)  \subseteq (r_jB_n\cap H(\overline{\va}, \alpha-|\overline{\va}|^2)\cap\Lambda).
\end{equation*}
Hence, we may assume that $\alpha_{r_j}\leq |\overline{\va}|^2$. 
Since $\alpha_r$ is integral, we even find a sequence of radii
$(r_j)_{j\in\N}\subseteq\N$ such that
$H(\va_{r_j},\alpha_{r_j})=H(\overline{\va},
\overline{\alpha})=:\overline{H}$ for all $j$ and a fixed 
$\overline{\alpha}\in\N$
.

We choose $K_j=r_jB_n$. In order to estimate the limit, we want to apply \eqref{eq:vol-appr} to $r_jB_n$ and $r_jB_n\cap\overline{H}$. The latter body
 may be viewed as a ball of radius $r_j - o(r_j)$ that is embedded in an
$(n-1)$-space together with a translation of
$\Lambda\cap\overline{\va}^\perp$. Thus, by \eqref{eq:vol-appr}, 
$$\begin{aligned}
& \lim_{j\rightarrow\infty} \frac{\#_\Lambda(r_jB_n)^{n-1}}{\max \#_\Lambda(r_jB_n\cap H)^n}\\
& = \lim_{j\rightarrow\infty} \frac{\big(\#_\Lambda (r_jB_n)/r_j^n\big)^{n-1}}{\big(\#_\Lambda(r_jB_n\cap\overline{H})/(r_j-o(r_j))^{n-1}\big)^n}\\
&= |\overline{\va}|^n \frac{\omega_n^{n-1}}{\omega_{n-1}^n} \geq (c\sqrt{n})^n e^{-c^\prime n},
\end{aligned}$$ 
where $c^\prime>0$ is an absolute constant. In the last step we used the assumption that
$\lambda_1(B_n,\Lambda)=c\sqrt{n}$ and Stirling's formula to estimate
the volumes, which are known to be $\pi^\frac{n}{2}/\Gamma(\frac{n}{2}+1)$, where $\Gamma(\cdot)$ is the Eulerian Gamma-function. Taking the $n$-th root yields the claim.
\end{proof}



\section{Discrete version of the reverse Loomis--Whitney inequality}
\label{sec:proj}

The goal of this section is to prove Theorem \ref{thm:drev_lw}. First, 
for a given $\vv\in\Z^n$  we have to estimate
the number of points  in the projection of $K|\vv^\perp$ with respect
to the projected lattice $\Z^n|\vv^\perp$ 
against
the number of  points in $(K\cap\Z^n)|\vv^\perp$.


\begin{lemma}
\label{lemma:preimages}
Let $K\in \Ko^n$ and $\vv\in (K\cap\Z^n)\setminus\{\vnull\}.$ Then
\begin{equation*} 
\LE_{\Z^n|\vv^\perp} (K|\vv^\perp)  \leq O(1)^n
\cdot|(K\cap\Z^n)|\vv^\perp|
\end{equation*} 
\end{lemma}

\begin{proof}
For short we write $\overline{K} = K|\vv^\perp$ and $\Lambda =
\Z^n|\vv^\perp$. Consider a line $\ell\subseteq\vv^\perp$ that contains at least 5 points from $\overline{K}\cap\Lambda$, i.e.,
\begin{equation*}
  \ell\cap\overline K\cap\Lambda=\{i\cdot \vy: -m\leq i\leq m\}
\end{equation*}
for some $\vy\in\Lambda\setminus\{\vnull\}$ and $m\geq 2$.
Since $\vv\in K$, the length of
$\R\vv\cap K$ is at least $2|\vv|$. Therefore, the length of the segment  $(i\cdot
\vy+\R\vv)\cap K$ is at least $|\vv|$, for any $|i|\leq\lfloor
\frac{m}{2}\rfloor$, as can be seen by considering the triangle
given by $\vv$, $-\vv$ and a point $\vw\in K$ with $\vw|\vv^\perp=m \vy$ (cf.\ Fig.\ \ref{fig:preimages}).
\begin{figure}[hbt]
\centering
\includegraphics[width = .55\textwidth]{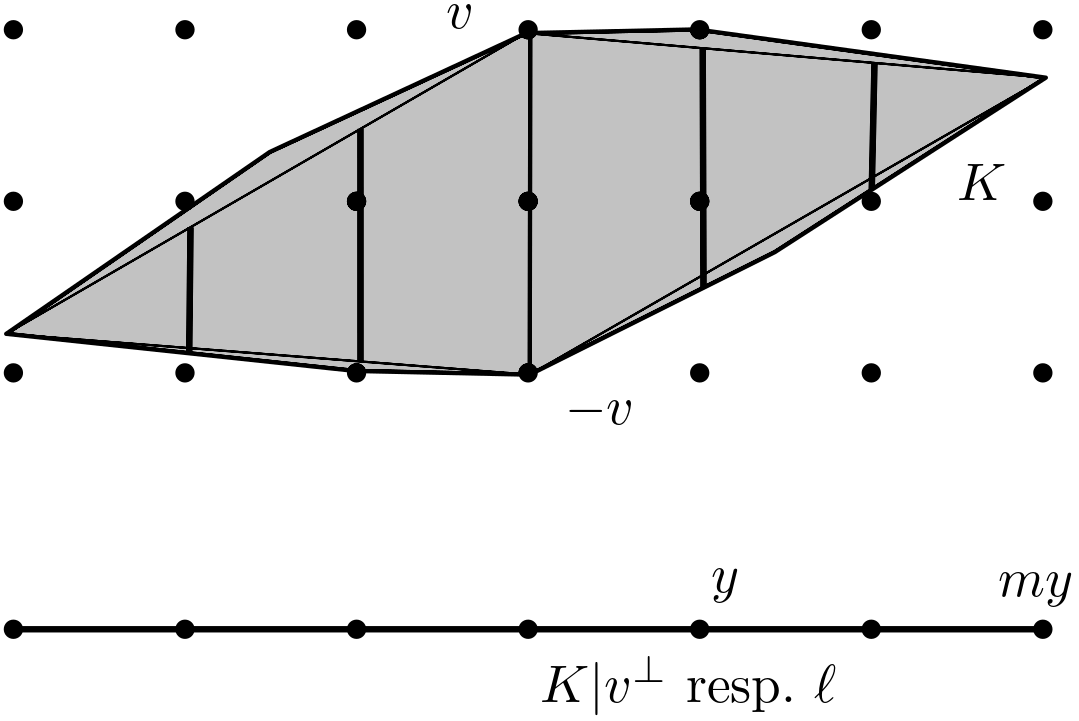}
\caption{The lattice points in the projection close to the origin admit a preimage in $K$ that is long enough to ensure a lattice point.}
\label{fig:preimages}
\end{figure}
So, as long as $i\leq m/2$ the section $(i\cdot\vy+\R\vv)$
contains a lattice point. Since $m\geq 2$,  at least
$1/3$ of the points  in $\ell\cap\Lambda\cap \overline{K}\setminus\{0\}$
have a preimage in $K\cap\Z^n$.
Hence, with 
\begin{equation*} 
  A = \big\{\vx\in \overline{K}\cap\Lambda\setminus\{\vnull\}:|\R \vx\cap
  \overline{K}\cap\Lambda|\geq 5\},
\end{equation*} 
we know that $|A|\leq 3\big|(K\cap\Z^n)|\vv^\perp\big|$ and so 
it suffices to prove
\begin{equation} 
  \LE_\Lambda(\overline{K}) \leq O(1)^n\,|A|.
  \label{eq:itsuffices}
\end{equation} 
  To this end, let $R_1,...,R_{4^{n-1}}\subseteq\Lambda$ be the cosets
  of $4\Lambda$ in $\Lambda$ and consider two distinct points
  $\vx,\vy\in\overline{K}\cap R_k$.
  Then  $\vx-\vy\in 4\Lambda$ and thus
\begin{equation*} 
  \frac{1}{2}(\vx-\vy)\in \overline{K}\cap 2\Lambda\setminus\{\vnull\}\subseteq
  A.
\end{equation*} 
Hence, $\frac{1}{2}((\overline{K}\cap R_k)-(\overline{K}\cap R_k))\subseteq A\cup\{0\}$ and thus
\begin{equation*}
\begin{aligned} 
|\overline{K}\cap R_k|&\leq \big|(\overline{K}\cap R_k)-(\overline{K}\cap R_k)\big|\\
&= \Big|\frac{1}{2}\big((\overline{K}\cap
R_k)-(\overline{K}\cap R_k)\big)\Big|\\
&\leq |A\cup\{\vnull\}|= |A| + 1.
\end{aligned}
\end{equation*}
If $A=\emptyset$, we immediately get
\begin{equation*} 
  \LE_\Lambda(\overline{K})\leq
  4^{n-1}\leq4^{n-1}\big|(K\cap\Z^n)|\vv^\perp\big|
\end{equation*}
and the claim of the lemma follows. So let $A\ne\emptyset$. Then
$|\overline{K}\cap R_k|\leq 2|A|$ which leads to
$|\overline{K}\cap\Lambda|\leq 4^{n-1}\cdot 2|A|$ and thus \eqref{eq:itsuffices}.
\end{proof}

\begin{remark}
\label{rem:projections}
The inequality of Lemma \ref{lemma:preimages} is essentially best-possible, in the sense that in any dimension, there is a convex body $K\in\Ko^n$ with $\ve_n\in K$, such that
\begin{equation*}
\LE_{\Z^n|\ve_n^\perp} (K|\ve_n^\perp) = 3^{n-1} \quad \mathrm{and} \quad (K\cap\Z^n)|\ve_n^\perp = \{\vnull\}.
\end{equation*}
To see this, let $\vu=(1,2,4,...,2^{n-1})^T\in\R^n$. We have $C_n\cap\vu^\perp\cap\Z^n =\{\vnull\}$; Suppose there was a non-zero point $\vx\in C_n\cap\vu^\perp\cap\Z^n$. Let $i$ be the largest index such $\vx_i\neq 0$. By symmetry, we may assume that $\vx_i>0$. It follows from 
\begin{equation}
\label{eq:2hochk}
\sum_{j=0}^{k-1} 2^j= 2^k-1,
\end{equation}
that $\langle \vx,\vu\rangle \geq 1$, a contradiction.

On the other hand, we have $(C_n\cap\vu^\perp)|\ve_n^\perp=C_{n-1}$; Let $\vx\in\{\pm 1\}^{n-1}$ be a vertex of $C_{n-1}$. Then, by \eqref{eq:2hochk}, $\big|\sum_{i=1}^{n-1} x_i2^{i-1}\big|\leq 2^{n-1}$. So there exists $x_n\in[-1,1]$ such that $(\vx,x_n)^T\in C_n\cap \vu^\perp$.

Thus, the convex body $K=\conv\big((C_n\cap \vu^\perp)\cup \{\pm \ve_n\}\big)$ has the desired properties.
\end{remark}

Now we are ready for the proof of Theorem \ref{thm:drev_lw}.
\begin{proof}[Proof of Theorem \ref{thm:drev_lw}]
  We abbreviate $\lambda_i=\lambda_i(K)$ and let
  $\vv_i\in\lambda_i\,K\cap\Z^n$, $1\leq i\leq n$,  be linearly
  independent. Due to our assumption we have $\lambda_n\leq 1$ and so   $\vv_i\in K$. So we may apply Lemma \ref{lemma:preimages} to obtain \begin{equation}\label{eq:rev_lw_eq1}
\prod_{i=1}^n \LE_{\Z^n|\vv_i^\perp}(K|\vv_i^\perp)\leq O(1)^{n^2}\prod_{i=1}^n \big|Z|\vv_i^\perp\big|,
\end{equation}
where $Z=K\cap\Z^n$. It is therefore enough to show
\begin{equation} 
  \prod_{i=1}^n \big| Z|\vv_i^\perp\big| \leq O(1)^{n^2} |Z|^{n-1}
  \label{eq:7}
\end{equation}   
 To this end we set $S_i = Z \cap \R \vv_i$, $1\leq i\leq n$. Then
 $|S_i| = 2\lfloor 1/\lambda_i\rfloor +1$. Now we choose a subset
 $Z_i\subseteq Z$ such that the projection $Z_i\rightarrow
 Z|\vv_i^\perp$ is bijective. Clearly, $Z_i+S_i\subseteq Z+Z$ and so we
 have 
\begin{equation*} 
\big| Z+Z\big|\geq  \big|Z_i+S_i\big| = (2\lfloor 1/\lambda_i\rfloor+1) \cdot
 \big|Z|\vv_i^\perp\big|\geq \frac{3}{4}(\lfloor 2/\lambda_i\rfloor
 +1)\big|Z|\vv_i^\perp\big|.
\end{equation*} 
So we obtain
\begin{equation*} 
|Z+Z|^n\geq  \Big(\frac{3}{4}\Big)^n \prod_{i=1}^n
\lfloor 2/\lambda_i+1\rfloor \prod_{i=1}^n \big|Z|\vv_i^\perp\big|.
\end{equation*} 
In view of Lemma \ref{lemma:elementary-trafos}, \eqref{eq:dilates} we
have $|Z+Z|\leq \#(2K)\leq 4^n\#K=4^n|Z|$ and thus
\begin{equation} 
4^{n^2} \Big(\frac{4}{3}\Big)^n|Z|^n\geq  \prod_{i=1}^n
\lfloor 2/\lambda_i+1\rfloor \prod_{i=1}^n \big|Z|\vv_i^\perp\big|.
\end{equation} 
Finally we use \eqref{eq:upper-bhw}, i.e., $\prod_{i=1}^n\lfloor
2/\lambda_i+1\rfloor \geq 3^{-n} |Z|$ in order to get \eqref{eq:7}.
\end{proof}

\begin{remark}\hfill
\label{rem:drev-lw}
\begin{enumerate}
\item The above proof does not depend on the particular properties of the lattice $\Z^n$. So one obtains the same statement for an arbitrary $n$-dimensional lattice $\Lambda\subseteq\R^n$. More precisely, if $K\in\Ko^n$ fullfills $\dim(K\cap\Lambda)=n$, we have
\begin{equation*}
\LE_\Lambda(K)^{\frac{n-1}{n}} \geq c^{-n} \Big(\prod_{i=1}^n \LE_{\Lambda|\vv_i^\perp}(K|\vv_i^\perp)\Big)^{\frac{1}{n}},
\end{equation*}
where $\vv_i\in\lambda_i(K,\Lambda)K\cap\Lambda$ are linearly independent.

\item Also, the above approach yields a reverse Loomis-Whitney-type inequality, if one aims for a lattice basis, instead of merely independent lattice vectors. 
However, in order to apply Lemma \ref{lemma:preimages} one has to ensure that the basis is contained in $K$. For the basis $\{\vb_1,...,\vb_n\}$ from equation \eqref{eq:thm:basis}, this means that one has to enlarge $K$ by a factor $n$. So in this case, we obtain
\begin{equation*}
\LE(K)^{\frac{n-1}{n}} \geq (cn)^{-n} \Big(\prod_{i=1}^n \LE_{\Z^n|\vv_i^\perp}(K|\vb_i^\perp)\Big)^{\frac{1}{n}}.
\end{equation*}
\end{enumerate}

\end{remark}


\section{Unconditional Bodies}
\label{sec:unconditional}

A convex body $K$ is called \emph{unconditional}, if it is symmetric
to all the coordinate hyperplanes, i.e., $(\pm x_1,...,\pm x_n)\in K$,
for any $\vx\in K$.  For such bodies we can improve some of our
inequalities.  We start with Lemma \ref{lemma:elementary-trafos}, \eqref{eq:dilates}

\begin{lemma}
\label{lemma:unconditional-dilates}
Let $K\in\K^n$ be unconditional and $m\in\N$. Then  $\#mK\leq (2m-1)^n \#K$ and the inequality is sharp.
\end{lemma}

\begin{proof}
First, we prove the claim for $\dim K=1$, i.e., we may assume that  $K=[-x,x]\subseteq\R$,
$x\geq 0$. Then, $\LE(K)=2\lfloor x\rfloor +1$. In case that $x\in\Z$ we have
\begin{equation*} 
  \# mK = 2mx+1\leq m(2x+1)= m\cdot \#K.
\end{equation*}
  So let $x\notin\Z$. Then, $$\#mK \leq 2mx+1 < 2m(\lfloor x\rfloor
  +1)+1.$$ Since both sides of the inequality are odd integers, we
  obtain 
  
  \begin{equation*}
  \begin{aligned}
  \#mK &\leq 2m(\lfloor x\rfloor +1)-1 = 2m\lfloor x\rfloor +2m-1\\
  & = (2m-1)\big(\frac{m}{2m-1}2\lfloor x\rfloor +1\big) \leq (2m-1)\LE K.
  \end{aligned}
  \end{equation*}

Next, let $K\subseteq\R^n$ be an arbitrary unconditional convex
body. Consider the unconditional body $K^\prime$ that we obtain by
multiplying the first coordinates in $K$ by $m$. The lattice points in
$K$ and $K^\prime$ can be partitioned into intervals parallel to
$e_1$. The intervals that we see in $K^\prime$ are exactly the
intervals of $K$, multiplied by $m$. So by the $1$-dimensional case  we have $\#K^\prime \leq (2m-1) \# K$. If we repeat this argument for every coordinate, we end up with the desired inequality. The cubes $K=\big[ -(1-\frac{1}{2m}),
1-\frac{1}{2m}\big]^n$   show that the inequality is sharp.
\end{proof}

We conjecture $(2m-1)^n$ to be the right constant also for arbitrary
symmetric convex bodies.  Lemma \ref{lemma:unconditional-dilates} yields a slightly improved version of Theorem \ref{thm:dmeyer} for the class of unconditional bodies.

\begin{proposition}
\label{prop:dmeyer-unconditional}
Let $K\in\Ko^n$ be unconditional. Then, $$\#K^{\frac{n-1}{n}}\geq \frac{1}{3^{n-1}} \Big(\prod_{i=1}^n \#(K\cap \ve_i^\perp)\Big)^{\frac{1}{n}}.$$
\end{proposition} 

\begin{proof} 
Again, we write $K_i = K\cap \ve_i^\perp$ and $h_i =
\suk(K,\ve_i)$. Note that $h_i$ is attained by a multiple of $\ve_i$,
since $K$ is unconditional. This implies
\begin{equation*} 
  K_i+[-h_i,h_i] \ve_i\subseteq 2K,
\end{equation*} 
and we obtain
\begin{equation*}
\#(2K)^{n-1} \geq \prod_{i=1}^{n-1}\big( (2\lfloor h_i\rfloor+1) \#K_i \geq
\prod_{i=1}^n \#K_i,
\end{equation*}

where the last inequality follows from $K_n\subseteq
[-h_1,h_1]\times...\times[-h_{n-1},h_{n-1}]$. The claim follows by
applying Lemma \ref{lemma:unconditional-dilates} to the left hand side above.
\end{proof}

Note that for an unconditional body $K\subseteq\R^n$ one has $K\cap
\ve_i^\perp = K|\ve_i^\perp$, $1\leq i\leq n$. Therefore, Proposition
\ref{prop:dmeyer-unconditional} is also a sharpening of Theorem
\ref{thm:drev_lw}. In fact, following the lines of the proof of the discrete reverse Loomis-Whitney inequality in Section \ref{sec:proj}, the above proof is a simplification of the proof in Section \ref{sec:proj}.

Moreover, the inequalities of Theorem
\ref{thm:drev_meyer} and \ref{thm:dslicing} hold with constant 1 for
unconditional bodies, by the Loomis--Whitney
inequality.

As for the discrete Brunn inequality, a constant 1 is obtained when intersecting an unconditional body $K$ with a coordinate subspace $L$, since every slice $K\cap(L+\vt)$ is mapped into the central slice injectively by the orthogonal projection onto $L$.  Moreover, for any hyperplane $H$ there is a coordinate $i$ such that the projection $H\rightarrow\ve_i^\perp$ is bijective and maps lattice points in $H$ to lattice points in $\ve_i^\perp$ (The index $i$ can be chosen to be an index for which the normal vector $\vv$ of $H$ is non-zero). Therefore, the maximal hyperplane section with respect to $\LE(\cdot)$ can always be chosen to be a coordinate section.

However, for general subspaces $L$ we cannot hope for a constant 1 in the discrete Brunn inequality, as the next example illustrates.

\begin{example}
\label{ex:dbrunn-unconditional}
Consider the symmetric cube $C_n=[-1,1]^n$ and the vector
$\vu=(1,2,4,...,2^{n-1})^T\in\R^n$. Then we have $\LE(C_n\cap \vu^\perp)=1$ (cf.\ Remark \ref{rem:projections}).

On the other hand, we have $\LE(C_n\cap\{\vx\in\R^n:\langle \vx,\vu\rangle=1\})= n$, because, in view of \eqref{eq:2hochk}, the points $\vx_k=\ve_k-\sum_{j=0}^{k-1} \ve_j$, $1\leq k\leq n$ are contained in this section and as in Remark \ref{rem:projections}, by considering the maximal non-zero coordinate of a lattice point in this section, there are no further points.
\end{example} 

\section{Acknowledgements}

We thank Eduardo Lucas Mar\'{i}n and Matthias Schymura for helpful comments and suggestions.


\bibliographystyle{plain}
\bibliography{dslicing}

\end{document}